\documentclass[review]{elsarticle}

\usepackage{hyperref}
\usepackage{amsmath}
\usepackage{amsthm}
\usepackage{amssymb}
\usepackage{etoolbox,color}
\usepackage{tikz}

\ifdefined\ShowComment
\def\mohit#1{\marginpar{$\leftarrow$\fbox{M}}\footnote{$\Rightarrow$~{\sf\textcolor{red}{#1 --Mohit}}}}
\def\sagnik#1{\marginpar{$\leftarrow$\fbox{S}}\footnote{$\Rightarrow$~{\sf\textcolor{red}{#1 --Sagnik}}}}
\def\danupon#1{\marginpar{$\leftarrow$\fbox{D}}\footnote{$\Rightarrow$~{\sf\textcolor{red}{#1 --Danupon}}}}
\def\aris#1{\marginpar{$\leftarrow$\fbox{A}}\footnote{$\Rightarrow$~{\sf\textcolor{red}{#1 --Aris}}}}
\else
\def\mohit#1{}
\def\sagnik#1{}
\def\danupon#1{}
\def\aris#1{}
\fi 

\usepackage[utf8]{inputenc}

\theoremstyle{plain}% default
\newtheorem{thm}{Theorem}[section]
\newtheorem{lem}[thm]{Lemma}
\newtheorem{prop}[thm]{Proposition}

\newcommand{\curly}[1]{\left\{ #1 \right\}}
\newcommand{\paren}[1]{\left( #1 \right)}

\DeclareMathOperator{\child}{children}

\DeclareMathOperator{\argmin}{argmin}
\DeclareMathOperator{\argmax}{argmax}

\newcommand{\parent}[2][]{%
	\ifstrempty{#1}{%
		\pi\paren{#2}
	}{%
		\pi_{{#1}}\paren{#2}
	}%
}

\newcommand{\anc}[2][]{%
	\ifstrempty{#1}{%
		{#2}^{\uparrow}
	}{%
		{#2}^{\uparrow {#1}}
	}%
}

\newcommand{\level}[2][]{%
	\ifstrempty{#1}{%
		\ell\paren{#2}
	}{%
		\ell_{#1}\paren{#2}
	}%
}

\newcommand{\treeedge}[2][]{%
	\ifstrempty{#1}{%
		{{e}}\paren{#2}
	}{%
		{{e}}_{{#1}}\paren{#2}
	}%
}

\newcommand{\CASE}[1]{\emph{CASE}-#1}

\newcommand{\desc}[2][]{%
	\ifstrempty{#1}{%
		{#2}^{\downarrow}
	}{%
		{#2}^{\downarrow {#1}}
	}%
}

\makeatletter
\def\ps@pprintTitle{%
  \let\@oddhead\@empty
  \let\@evenhead\@empty
  \let\@oddfoot\@empty
  \let\@evenfoot\@oddfoot
}
\makeatother
\begin{document}

\begin{frontmatter}

\title{Edge-Cuts and Rooted Spanning Trees}

%\author{Mohit Daga\fnref{myfootnote}}

\author[mysecondaryaddress]{Mohit Daga\corref{mycorrespondingauthor}}
\cortext[mycorrespondingauthor]{Corresponding author}
\ead{mdaga@kth.se}

%\address[mymainaddress]{}
\address[mysecondaryaddress]{KTH Royal Institute of Technology, Stockholm - Sweden}

\begin{abstract}
We give a closed form formula to determine the size of a $k$-respecting cut. Further, we show that for any $k$, the size of the $k$-respecting cut can be found only using the size of $2$-respecting cuts.
\end{abstract}


\end{frontmatter}

%TEX root = ../main-charec-lemma.tex
\section{Introduction}
An edge-cut of a graph is said to $k$ respect a given spanning tree, if the cut shares $k$ edges with the tree. The technique of finding cuts that k-respect a given set of spanning trees is used in designing algorithms to find the size of cuts. The pioneering use appears in two breakthrough results by Karger (ACM STOC 1996 and JACM 2000) and Thorup (ACM STOC 2001 and Combinatorica 2007). The former \cite{karger2000minimum} gives the first linear time algorithm to find the size of a min-cut, whereas the later \cite{thorup2007fully} gives the first fully dynamic algorithm for min-cuts. A common technique among these is to find the size of a minimum 2-respecting cut in a given set of spanning trees. Over the years this technique of finding a 2-respecting cut have found applications in designing algorithms to find min-cuts in several different settings and computational models: centralized, parallel, distributed, streaming, and dynamic.
 
In the centralized setting post the breakthrough result by Karger, recent results by Kwarabayashi and Thorup \cite{kawarabayashi2015deterministic}, improved further by Henzinger et al. \cite{henzinger2020local}  give a deterministic linear time. Several simpler algorithms have been designed that find the size of a min cut using new algorithms to find the size of a 2-respecting cut given a set of trees \cite{bhardwaj2020simple, gawrychowski2020minimum, saranurak2021simple}. Further, a recent breakthrough result that finds all pairs max flow in $\tilde O(n^2)$ uses $4$-respecting cuts \cite{Abboud2021-cd}. \mohit{this is example macro} \aris{this is example macro}.

In the distributed setting the first sub linear algorithm \cite{daga2019distributed} uses the concept of 2-respecting cut. Here the algorithms to find the min-cut has two parts: algorithm to reduce the size of the graph through a contraction mechanism and given a set of trees provide efficient algorithm to find the size of a 2-respecting cut. The result in  \cite{daga2019distributed} is further improved by providing better algorithm for one of the two parts, even leading to an optimal algorithm to find the size of min-cut \cite{dory2021distributed, ghaffari2022universally, ghaffari2020faster, mukhopadhyay2020weighted}. Though, the underlying similarity of finding the size of a minimum 2-respecting cut is common among all. There are several open problems that still remain here, for example the size of a small cut \cite{pritchard2011fast, parter2019small}, for a constant $k$.

Most of the aforementioned results rely on a closed form expression to find size of 2-respecting cuts, and the fact that only a small number of trees are required to be constructed in order to find a tree that 2-respects the minimum cut. The guarantees regarding the small number of trees come from Nash Williams, which states that the number of disjoint trees in a $k$-connected graph is at most $k/2$ \cite{chen1994short,diestel2018graph}. In this paper, we extend the closed form expression to find size of any k-respecting cut. Furthermore, we show that the size of any k-respecting cut can just be found using the size of 2-respecting cuts. Our results rely on the cut-space concept from graph-theory. First,  we give a closed form expression that finds the size of any k-respecting cut (Theorem \ref{thm:main_charac_lemma}). Secondly, we show that the size of any k-respecting cut can be find if we know the size of $1$-respecting cuts, and $2$-respecting cuts (Theorem \ref{thm:main_sec_2}). 

Let $G = (V,E)$ be the given tree. Given a rooted spanning tree T, let $E(T)$ be the edges in the tree, and let $\treeedge[T]{v}$, be the tree edge between $v$ and its parent for all $v \in V$, except the root. For any $A \subset V$, let $\delta(A)$ be the edges in the cut $(A, V\setminus A)$. For any rooted spanning tree $T$, let $\desc[T]{v}$ denote the set of vertices that are decedents of $v$ in $T$, including $v$ itself.
 \begin{thm}
	Let $G= (V,E)$ be a given graph, let $T$ be a rooted spanning tree of $G$ and $A\subset V$. Suppose 
	$E(T) \cap \delta(A) = \curly{\treeedge[T]{v_1},\ldots,\treeedge[T]{v_k}}$, 
	for some vertices $S = \curly{v_1,\ldots,v_k}$. Then                                                                                                  
	\begin{align}
	\label{eqn:cut}
	|\delta(A)| = \sum_{l=1}^k(-1)^{l-1}2^{l-1}\sum_{S'\subseteq[k]\atop|S'|=l}\left|\bigcap_{i\in S'}\delta(\desc[T]{v_i})\right|
	\end{align}	
	\label{thm:main_charac_lemma}
\end{thm}
Further using some combinatorial arguments we show that the size of any $k$-respecting cut can just be found using pair-wise 2-respecting cuts. 
\begin{thm}
	\label{thm:main_sec_2}		
	Let $G= (V,E)$ be a given graph, 
	let $T$ be a rooted spanning tree of $G$ and $A\subset V$. 
	Suppose 
	$E(T) \cap \delta(A) = \curly{\treeedge[T]{v_1},\ldots,\treeedge[T]{v_k}}$, 
	for some vertices $S = \curly{v_1,\ldots,v_k}$.
	Then $|\delta(A)|$ can be determined if the following is known 
	\begin{itemize}
		\item $|\delta(\desc[T]{x})|\ \forall x \in S$,
		\item $|\delta(\desc[T]{x}) \cap \delta(\desc[T]{y})|\ \forall x,y \in S$ and
		\item the path from root of $T$ to $x$ for all $x \in S$.
	\end{itemize} 
\end{thm}

%TEX root = ../main-charec-lemma.tex
\section{Preliminaries}
For any rooted spanning tree $T$, we denote $r_T$ as its root. For any $v \in V$, let $\desc[T]{v}$ be the vertex set that are decedents of $v$ in the tree $T$ including itself. Similarly, $\anc[T]{v}$ is the set of vertices which are on the path $r_T$ to $v$. For all vertices $v \in V\setminus r_T$, we use $\parent[T]{v}$ to denote the parent of $v$ in $T$ and $\treeedge[T]{v}$ to denote the tree edge between $v$ and $\parent[T]{v}$. Let $\level[T]{v}$ be the distance of vertex $v$ from the root $r_T$ in the tree $T$. Let $\child_T(v)$ denote the children of the vertex $v$ in the tree $T$. If $v$ is a leaf node, then define $\child_T(v)$ to be $\emptyset$. For any two vertices $v$ and $u$, we say that $v$ and $u$ are \emph{independent w.r.t the tree $T$}, denoted using $v\perp_T u$ iff $\desc[T]{v} \cap \desc[T]{u} = \emptyset$. If they are not then we say that they are not independent, and denote it using $v\not\perp_T u$. 

We will use $\oplus$ to denote the symmetric difference operator. More precisely, for any sets $A_1,A_2,\ldots,A_k$, we have $a\in\bigoplus_{k'=1}^k A_{k'}$ iff $|\{k'\in[k]:a\in A_{k'}\}|$ is odd. Throughout this paper, when we use $\curly{}$, we mean it to be set, and not multi-set. That is each entry in $\{\}$ occurs exactly once. Whenever we use an index $k$, it means a whole number less than the number of vertices in the graph.

We stress the readers to familiarize with $\oplus$ operator. We make the following simple preposition to do the same. These qualify how the symmetric difference operator appears when two sets are considered.
\begin{prop}
	Let $A_1$ and $A_2$ be any two sets. Suppose $A_1 \cap A_2 = \emptyset$, then $A_1 \oplus A_2 = A_1 \cup A_2$. Further, suppose, $A_2 \subseteq A_1$, then $A_1 \oplus A_2 = A_1 \setminus A_2$. 
	\label{simple-prop}
\end{prop}

We use a well-known result that states that cut-spaces are a vector space with respect to the $\oplus$ operator. This can be found in standard graph theory books, for ex. Bondy and Murty \cite{bondy1976graph}. 
\begin{lem}[Also noted as Proposition 2.1 in \cite{pritchard2011fast}]	Let $T$ be a given spanning tree and $v_1,v_2,\ldots,v_k \in V$. Then $\delta(\desc[T]{v_1} \oplus \desc[T]{v_2} \oplus \ldots \oplus \desc[T]{v_k}) = \delta(\desc[T]{v_1}) \oplus \delta(\desc[T]{v_2}) \oplus \ldots \oplus \delta(\desc[T]{v_k})$
	\label{cutspace-vector-space}
\end{lem}

We also mention a set-theoretic result for the cardinality of xor operation of $k$ sets.
\begin{prop}
	Suppose $A_1,\ldots,A_k$ are some sets. Then $$\left|\bigoplus_{i=1}^kA_i\right|=\sum_{l=1}^k(-1)^{\ell-1}2^{l-1}\sum_{S\subseteq[k]\atop|S|=l}\left|\bigcap_{i\in S}A_i\right|.$$
	\label{set-theory-result}
\end{prop}

%TEX root = ../main-charec-lemma.tex
\section{Cut Characterization Lemma \label{cut-charec-proof}}
In this section, we prove the characterization given in Theorem \ref{thm:main_charac_lemma}. We know that $\delta(A) = \delta(V\setminus A)$. We show that for any $A\subset V$, if $\delta(A) \cap E(T) = \curly{\treeedge[T]{v_1},\ldots,\treeedge[T]{v_k}}$ for some vertex set $S = \curly{v_1,\ldots,v_k}$, where $v_i \in V\setminus r_T$, then either $A = \desc[T]{v_1}\oplus\ldots\oplus\desc[T]{v_k}$ or $V\setminus A = \desc[T]{v_1}\oplus\ldots\oplus\desc[T]{v_k}$. This together with Lemma \ref{cutspace-vector-space} and Proposition \ref{set-theory-result} leads to Theorem \ref{thm:main_charac_lemma}.
\begin{prop}
	Let $T$ be a rooted spanning tree. For any vertex set $A \subset V$, either $A$ or $V\setminus A$ is equal to  $\bigoplus_{v\in S} \desc[T]{v}$for some $S \subseteq V\setminus r_T$.
	\label{prop-A-can-be-written}
\end{prop}
\begin{proof}
	Let the root $r_T \notin A$. For any $v \in A$, we know that for any $c \in \child_T(v)$, $\desc[T]{c} \subset \desc[T]{v}$. Using Proposition \ref{simple-prop}, $\curly{v} = \desc[T]{v} \bigoplus\limits_{c \in \child_T(v)}\desc[T]{c}$. Thus,
	\begin{align*}
	A &= \bigcup\limits_{v \in A}{(\desc[T]{v} \bigoplus\limits_{c \in \child_T(v)}\desc[T]{c})}\\ 
	&= \bigoplus\limits_{v \in A}\paren{\desc[T]{v} \bigoplus\limits_{c \in \child_T(v)}\desc[T]{c}} 
	\end{align*}
	Here the last equality is true because for any two sets $S_1,S_2$ if $S_1 \cap S_2  = \emptyset$, then $S_1 \cup S_2 = S_1 \oplus S_2$ (see Proposition \ref{simple-prop}).  For avoiding multiplicity of occurrences and enforcing $S$ to be a set, we can remove any vertex $v$, if $\desc[T]{v}$ occurs even number of times, and keep it only once when it occur odd number of times. When $r_T \in A$, then $V\setminus A$ can be represented similarly.
\end{proof}
\begin{prop}
For any $u,v \in V\setminus r_T$, if $u \neq v$ and $u\in \desc[T]{v}$, then $\parent[T]{u}\in \desc[T]{v}$.
\label{prop-both-u-parent-u-in-desc-v}
\end{prop}
\begin{prop}
For any, rooted spanning tree $T$ and $v \in V\setminus r_T$, $\delta(\desc[T]{v}) \cap E(T) = \treeedge[T]{v}$
\label{prop-only-tree-edge}
\end{prop}
\begin{proof}
By definition, $v \in \desc[T]{v}$ and $\parent[T]{v}\notin \desc[T]{v}$, hence $\treeedge[T]{v} \in \desc[T]{v}$. Now for this proof, we need to show that for $u \neq v$, $\treeedge[T]{u} \notin \desc[T]{v}$. Suppose not. Then  both $u,\parent[T]{u}$ are not in $\desc[T]{v}$ simultaneously. But this cannot be true (see Proposition \ref{prop-both-u-parent-u-in-desc-v}). 
\end{proof}
Now we give a proposition that finds the tree edges in  $\delta(\desc[T]{v_1} \oplus \desc[T]{v_2} \oplus\ldots \oplus \desc[T]{v_k})$, for any $v_1,v_2,\ldots,v_k \in V$.
\begin{prop}
	Let $T$ be a rooted spanning tree. Let $S = \curly{v_1,\ldots,v_k} \subset V \setminus r_T$. Then $E(T) \cap  \delta(\desc[T]{v_1}  \oplus \desc[T]{v_2} \oplus \cdots \oplus \desc[T]{v_k}) = \curly{\treeedge[T]{v_1},\ldots,\treeedge[T]{v_k}}$.
	\label{claim-8a}
\end{prop}
\begin{proof}
From Lemma \ref{cutspace-vector-space}, $\delta(\desc[T]{v_1}  \oplus \desc[T]{v_2} \oplus \cdots \oplus \desc[T]{v_k}) = \delta(\desc[T]{v_1} )   \oplus \delta( \desc[T]{v_2})   \oplus \cdots  \oplus \delta(\desc[T]{v_k})$. From Proposition \ref{prop-only-tree-edge}, $\delta(\desc[T]{v}) \cap E(T) = \treeedge[T]{v}$. Hence, \sloppy$\curly{\treeedge[T]{v_1},\ldots,\treeedge[T]{v_k}}$ are the only tree edge that survives in $\delta(\desc[T]{v_1} )   \oplus \delta( \desc[T]{v_2})   \oplus \ldots  \oplus \delta(\desc[T]{v_k})$ because each one of these occurs only once, and no other edge in $E(T)$ is in the set $\delta(\desc[T]{v_i})$.
\end{proof}

%\begin{prop}
%	Let $T$ be a rooted spanning tree. Let $S = \curly{v_1,\ldots,v_k} \subset V \setminus r_T$. Then 
%	$\curly{\treeedge[T]{v_1},\ldots,\treeedge[T]{v_k}} \subseteq \delta(\desc[T]{v_1}) \oplus \delta(\desc[T]{v_2}) \oplus\ldots \oplus \delta(\desc[T]{v_k})$. Further, $\curly{\treeedge[T]{v_1},\ldots,\treeedge[T]{v_k}}$ are the only tree edges which are in $\delta(\desc[T]{v_1}) \oplus \delta(\desc[T]{v_2}) \oplus\ldots \oplus \delta(\desc[T]{v_k})$.
%	\label{claim-8}
%\end{prop}
%%\Rev{I don't understand the reason why Claim 7 holds. In line 125, it shows that $A = \oplus_{v \in A}(v^{\downarrow T} \oplus_{c\in children_{T}(v) c^{\downarrow T}})$ holds. Therefore, $A = \oplus_{v \in A} v$ holds for using Claim 6. However, it does not mean that$A = v_{1}^{\downarrow T} \oplus ... \oplus v_{k}^{\downarrow T}$ always holds.}
%%\mohit{add the reasoning. make it more clear.}
%\begin{proof}
%	For any $u \in V\setminus r_T$, the *only* tree edge in $\delta(\desc[T]{u})$ is $\treeedge[T]{u}$. Hence in the symmetric difference of the sets $\delta(\desc[T]{v_1}), \delta(\desc[T]{v_2}), \ldots, \delta(\desc[T]{v_k})$, the edge set $\curly{\treeedge[T]{v_1},\ldots,\treeedge[T]{v_k}}$ survives. This is because each one of these edges are only in one set. Also, these are all the tree edges in $\delta(\desc[T]{v_1}) \oplus \delta(\desc[T]{v_2}) \oplus\ldots \oplus \delta(\desc[T]{v_k})$ because no other tree edge is in the sets whose symmetric difference is considered.
%\end{proof}
\begin{prop}
For any $A\subset V$, if $\delta(A) \cap E(T) = \{\treeedge[T]{v_1},\treeedge[T]{v_2},\ldots,\treeedge[T]{v_k}\}$, then either A, or $V\setminus A$ is equal to $v_1^\downarrow \oplus  v_2^\downarrow \oplus \cdots \oplus  v_k^\downarrow$.
\label{prop-sec3-main}
\end{prop}
\begin{proof}
We know that either $A$, or $V\setminus A$ can be written as $\bigoplus_{v\in S} \desc[T] {v}$ (see Proposition \ref{prop-A-can-be-written}).  Also, $\delta(A) = \delta(V\setminus A)$. We claim that $S = \curly{v_1,v_2,\ldots,v_k}$, because if it is not, then $\delta(A) \cap E(T) \neq \{\treeedge[T]{v_1},\treeedge[T]{v_2},\ldots,\treeedge[T]{v_k}\}$ (Proposition \ref{claim-8a}).
\end{proof}
\begin{proof}[Proof of Theorem \ref{thm:main_charac_lemma}]
According to the given condition, and Proposition \ref{prop-sec3-main} $\delta(A) = \delta(v_1^\downarrow \oplus  v_2^\downarrow \oplus \cdots \oplus  v_k^\downarrow) = \delta(\desc[T]{v_1}) \oplus \delta(\desc[T]{v_2}) \oplus \cdots \oplus \delta(\desc[T]{v_k})$. Using Proposition \ref{set-theory-result} concludes the proof.
\end{proof}
%\begin{lemma}
%	Let $T$ be a spanning tree and $A\subset V$. Suppose for some vertices $S = \curly{v_1,\ldots,v_k}$, $\curly{\treeedge[T]{v_1},\ldots,\treeedge[T]{v_k}}$ are all the tree edges which are part of the cut-set $\delta(A)$. Then                                                                                                  
%	\begin{align*}
%	|\delta(A)| = |\delta(\desc[T]{v_1} \oplus \desc[T]{v_2} \oplus\ldots \oplus \desc[T]{v_k})|  &= 
%	|\delta(\desc[T]{v_1}) \oplus \delta(\desc[T]{v_2}) \oplus\ldots \oplus \delta(\desc[T]{v_k})|  \\
%	&= \sum_{l=1}^k(-1)^{l-1}2^{l-1}\sum_{S\subseteq[k]\atop|S|=l}\left|\bigcap_{i\in S}\delta(\desc[T]{v_i})\right|
%	\end{align*}
	
%	Expanding the above, we have
%	\begin{align*}
%	|\delta(\desc[T]{v_1}) \oplus \delta(\desc[T]{v_2}) \oplus\ldots \oplus \delta(\desc[T]{v_k})| 
%	&= \sum\limits_{\forall i_1, 1 \leq i_1 \leq k} |\delta(\desc[T]{v_{i_1}})|\\
%	&- 2\sum\limits_{\forall i_1,i_2\ 1 \leq i_1 < i_2 \leq k} |\delta(\desc[T]{\curly{v_{i_1}}})\cap \delta(\desc[T]{\curly{v_{i_2}}})| \\
%	& + 4\sum\limits_{\forall i_1,i_2,i_3\ 1 \leq i_1 < i_2 < i_3 \leq k} |\delta(\desc[T]{\curly{v_{i_1}}})\cap\delta(\desc[T]{\curly{v_{i_2}}})\cap \delta(\desc[T]{\curly{v_{i_3}}})|\\
%	&- 8 \ldots\\
%	\end{align*}
%	\label{lemma:main_charac_lemma}
%\end{lemma}
\section{Proof of Theorem \ref{thm:main_sec_2}}
We prove Theorem \ref{thm:main_sec_2} in this section. For any vertex sets $A_1,A_2,A_3,\ldots A_i \subseteq V$ we define $\gamma(A_1,A_2,\ldots,A_i) \triangleq	 |\delta(A_1) \cap \delta(A_2)\cap\ldots\cap\delta(A_i)|$. 
For example $\gamma(A) = |\delta(A)|$, $\gamma(A_1,A_2) = |\delta(A_1) \cap \delta(A_2)|$, $\gamma(A_1,A_2,A_3) = |\delta(A_1)\cap \delta(A_2) \cap \delta(A_3)|$. Note that $\gamma(\cdot)$ is an overloaded function. 

The issue with Theorem \ref{thm:main_charac_lemma} is the involvement of \emph{$k$-wise gamma} values, that is $\gamma(A_1,\ldots,A_k)$ for $k$ different vertex sets. We give the following lemma which states that pair-wise gamma values are enough to compute these $k$-wise gamma values, when our vertex sets are derived from a spanning tree $T$, and are of the form $\desc[T]{x}$ for any $x \in V$. This combined with Theorem \ref{thm:main_charac_lemma} proves Theorem \ref{thm:main_sec_2}.
\begin{lem}
Let $T$ be any spanning tree. Let $S = \curly{x_{1},x_{2},\ldots,x_{k}}$ be a subset of $V\setminus \{r_T\}$, where $k\geq 2$. Then $\gamma(\desc[T]{{x_{1}}},\desc[T]{{x_{2}}},\ldots,\desc[T]{{x_{k}}})$ is either equal to $0$ or $\gamma(\desc[T]{x_i},\desc[T]{x_j})$, for some $x_i,x_j \in S$. 	\label{lemma:pair_wise_info_is_enough}
\end{lem}
We describe the proof in this section. 
We give four exhaustive cases based on the ancestor-descendent relationships of the vertices in S with respect to the spanning tree $T$. They are shown in Figure. \ref{fig-four-different-cases}.
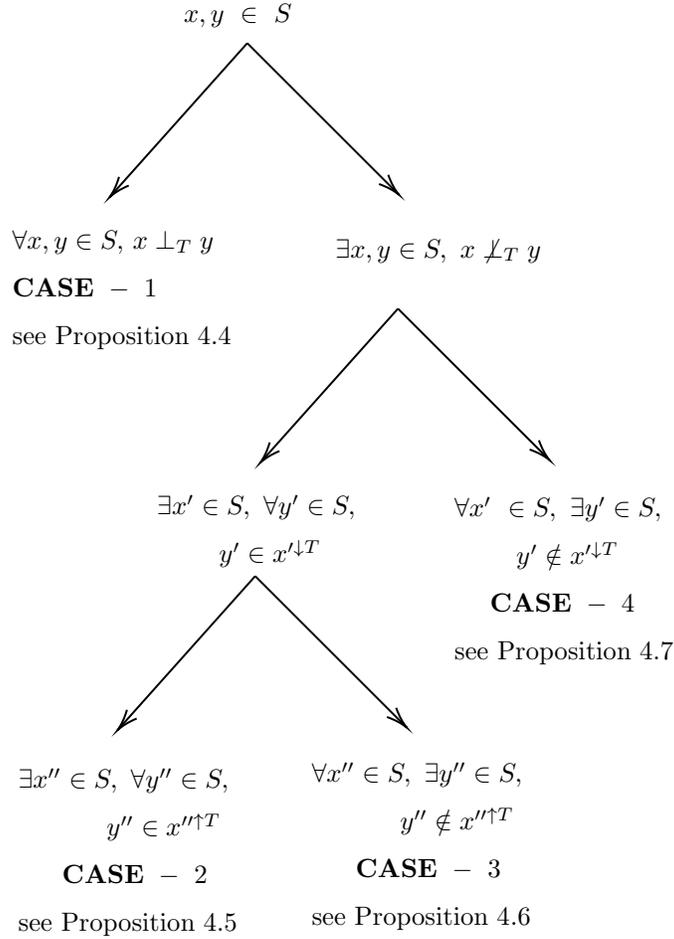
\begin{figure}[ht]
\centering
\tikzset{every picture/.style={line width=0.75pt}} %set default line width to 0.75pt        

\begin{tikzpicture}[x=0.75pt,y=0.75pt,yscale=-1,xscale=1]
%uncomment if require: \path (0,472); %set diagram left start at 0, and has height of 472

%Straight Lines [id:da14601044999783652] 
\draw    (129.5,26.66) -- (61.33,103) ;
\draw [shift={(60,104.49)}, rotate = 311.76] [color={rgb, 255:red, 0; green, 0; blue, 0 }  ][line width=0.75]    (10.93,-3.29) .. controls (6.95,-1.4) and (3.31,-0.3) .. (0,0) .. controls (3.31,0.3) and (6.95,1.4) .. (10.93,3.29)   ;
%Straight Lines [id:da007022079985875562] 
\draw    (129.5,26.66) -- (204.59,101.74) ;
\draw [shift={(206,103.16)}, rotate = 225] [color={rgb, 255:red, 0; green, 0; blue, 0 }  ][line width=0.75]    (10.93,-3.29) .. controls (6.95,-1.4) and (3.31,-0.3) .. (0,0) .. controls (3.31,0.3) and (6.95,1.4) .. (10.93,3.29)   ;
%Straight Lines [id:da47762347902577407] 
\draw    (205.5,160.66) -- (137.33,237) ;
\draw [shift={(136,238.49)}, rotate = 311.76] [color={rgb, 255:red, 0; green, 0; blue, 0 }  ][line width=0.75]    (10.93,-3.29) .. controls (6.95,-1.4) and (3.31,-0.3) .. (0,0) .. controls (3.31,0.3) and (6.95,1.4) .. (10.93,3.29)   ;
%Straight Lines [id:da18957433415521208] 
\draw    (205.5,160.66) -- (280.59,235.74) ;
\draw [shift={(282,237.16)}, rotate = 225] [color={rgb, 255:red, 0; green, 0; blue, 0 }  ][line width=0.75]    (10.93,-3.29) .. controls (6.95,-1.4) and (3.31,-0.3) .. (0,0) .. controls (3.31,0.3) and (6.95,1.4) .. (10.93,3.29)   ;
%Straight Lines [id:da45694531231397373] 
\draw    (133.5,295.66) -- (65.33,372) ;
\draw [shift={(64,373.49)}, rotate = 311.76] [color={rgb, 255:red, 0; green, 0; blue, 0 }  ][line width=0.75]    (10.93,-3.29) .. controls (6.95,-1.4) and (3.31,-0.3) .. (0,0) .. controls (3.31,0.3) and (6.95,1.4) .. (10.93,3.29)   ;
%Straight Lines [id:da1203672149535695] 
\draw    (133.5,295.66) -- (208.59,370.74) ;
\draw [shift={(210,372.16)}, rotate = 225] [color={rgb, 255:red, 0; green, 0; blue, 0 }  ][line width=0.75]    (10.93,-3.29) .. controls (6.95,-1.4) and (3.31,-0.3) .. (0,0) .. controls (3.31,0.3) and (6.95,1.4) .. (10.93,3.29)   ;

% Text Node
\draw (96,5.4) node [anchor=north west][inner sep=0.75pt]    {${\textstyle x,y\ \in \ S}$};
% Text Node
\draw (3,112.4) node [anchor=north west][inner sep=0.75pt]    {$ \begin{array}{l}
{\textstyle \forall x,y\in S,}
{\textstyle\ x\perp_T y}\\
{\textstyle \mathbf{CASE\ } -\ 1}\\
\text{see Proposition \ref{obs:gamma_1}}
\end{array}$};
% Text Node
\draw (166,116.4) node [anchor=north west][inner sep=0.75pt]    {$ \begin{array}{l}
\exists x,y\in S,\ x\not\perp_T y
\end{array}$};
% Text Node
\draw (76,246.4) node [anchor=north west][inner sep=0.75pt]    {$ \begin{array}{l}
\exists x'\in S,\ \forall y'\in S,\ \\
\ \ \ \ \ \ \ y'\in x^{\prime \downarrow T }
\end{array}$};
% Text Node
\draw (226,247.4) node [anchor=north west][inner sep=0.75pt]    {$ \begin{array}{l}
\forall x'\ \in S,\ \exists y'\in S,\\
\ \ \ \ \ \ \ y'\notin x^{\prime \downarrow T}\\
\ \ \ \ \mathbf{CASE} \ -\ 4\\
\text{see Proposition \ref{obs:gamma_4}}
\end{array}$};
% Text Node
\draw (6,384.4) node [anchor=north west][inner sep=0.75pt]    {$ \begin{array}{l}
\exists x''\in S,\ \forall y''\in S,\ \\
\ \ \ \ \ \ \ \ \ \ y''\in x'^{\prime \uparrow T}\\
\ \ \ \ \ \mathbf{CASE} \ -\ 2\\
\text{see Proposition \ref{obs:gamma_2}}
\end{array}$};
% Text Node
\draw (154,381.4) node [anchor=north west][inner sep=0.75pt]    {$ \begin{array}{l}
\forall x''\in S,\ \exists y''\in S,\ \\
\ \ \ \ \ \ \ \ \ \ y''\notin x'^{\prime \uparrow T}\\
\ \ \ \ \ \mathbf{CASE} \ -\ 3\\
\text{see Proposition \ref{obs:gamma_3}}
\end{array}$};
\end{tikzpicture}
\caption{The figure represents the four different cases.}
\label{fig-four-different-cases}
\end{figure}
We show that 
the {$k$ wise gamma} value is zero in two of these cases: \CASE{1} and \CASE{3}. 
Also, in the remaining two cases: \CASE{2} and \CASE{4}, the {$k$ wise gamma} value can be written in terms of pair-wise gamma value and {$k-1$ wise gamma} value respectively, for any $k\geq 3$. Using recursion, this implies that {$k$ wise gamma} value can be found from {pair-wise gamma} value
if we know the ancestor-decedent relationship of the vertices. 

The cases in  Figure \ref{fig-four-different-cases} are based on the ancestor-decedent relationship of all $x,y \in S$. Recall that $x\perp_T y$, if $\desc[T]{x} \cap \desc[T]{y} = \emptyset$. We show the following simple observation. This shall enable the reader to get familiarized with the $\perp_T$ operator.
\begin{prop}
	Let $T$ be a rooted spanning tree. Let $x,y$ be any nodes. If $x \not\perp_T y$, then either $\desc[T]{x} \subset \desc[T]{y}$ or $\desc[T]{y} \subset \desc[T]{x}$. Further, if $x\not\perp_T y$, and $\level[T]{x} < \level[T]{y}$, then $\desc[T]{y} \subset \desc[T]{x}$.
	\label{prop-simple}
\end{prop}
\begin{proof}
If $x\not\perp_T y$,  then  $\desc[T]{x} \cap \desc[T]{y} \neq \emptyset$. 
 The set 
 $\desc[T]{x}$ contains $x$ 
 and its decedents in the rooted spanning tree $T$. If the intersection of $\desc[T]{x}$ and $\desc[T]{y}$ is not empty, it means that $x$ and $y$ have an ancestor-decedent relationship in the tree $T$. Thus $\desc[T]{x} \subset \desc[T]{y}$ or $\desc[T]{y} \subset \desc[T]{x}$. When $\level[T]{x} < \level[T]{y}$, 
 then $y$ is a descendant of $x$, hence $\desc[T]{y} \subset \desc[T]{x}$
\end{proof}

We describe the four exhaustive cases from Figure \ref{fig-four-different-cases}.  In \CASE{1},  $x \perp_T y$ $\forall x,y\in S$. 
In Proposition \ref{obs:gamma_1}, 
we prove
\sloppy$\gamma(\desc[T]{{x_{1}}},\desc[T]{{x_{2}}},\ldots,\desc[T]{{x_{k}}}) = 0$ for this case. 
\CASE{2,3,4} are negation of \CASE{1}. In these cases, $\exists x,y$ such that $x \not\perp_T y$. Further, these cases are divided into two groups based on the statement: $\exists x'\in S$  such that $\forall y'\in S$, $y' \in x'^{\downarrow T}$ ($y'$ is a descendent of $x'$). If this statement is false then it is \CASE{4}, if it is true, then it is one of \CASE{2} or \CASE{3}. For \CASE{4}, in Proposition \ref{obs:gamma_4} we show that k-wise gamma value can be found through $k-1$ wise gamma value.
Lastly, \CASE{2,3} distinguish between each other based on the statement:  $\exists x'' \in S$ such that all $y''\in S$ are on the tree path from $r_T$ to $x''$ ($y'' \in x''^{\uparrow T}$). If this is true, then this is \CASE{2}, and when it is false it is \CASE{3}. In Proposition \ref{obs:gamma_2}, we show that for \CASE{2}, k-wise gamma value can be found through pair-wise gamma value. For \CASE{3}, in Proposition \ref{obs:gamma_3}, we show that the k-wise gamma value is 0. 
Based on the aforementioned discussion, we state the following lemma.
\begin{lem}
Let $T$ be any spanning tree. Let $S = \curly{x_{1},x_{2},\ldots,x_{k}}$ be a subset of $V\setminus \{r_T\}$, where $k\geq 2$. If we know $\gamma(\desc[T]{x_i},\desc[T]{x_j})$ for all $x_i,x_j \in S$, and the ancestor decedent relationship between $x_i$ and $x_j$ for all $x_i,x_j \in S$, then $\gamma(\desc[T]x_{1},\desc[T]x_{2},\ldots,\desc[T]x_{k})$ can be found.
\end{lem}

%In all the remaining cases
%	\item $\exists x,y \in S$ such that $x \not\perp_T y$
%	\begin{itemize}
%		\item $\exists x \in S$ such that $\forall y \in S\setminus\curly{x}$, $y \in \desc[T]{x}$
%			\begin{itemize}
%				\item $\exists x \in S$, 
%				such that $\forall y \in S$, $y$ 
%				belongs on the tree path between $r_T$ to $x$. (\CASE{2}, dealt in Proposition \ref{obs:gamma_2})
%				\item $\forall x \in S$, 
%				$\exists y\in S$, such that $y$ does not
%				belong on the tree path between $r_T$ to $x$. (\CASE{3}, dealt in Proposition \ref{obs:gamma_3})
%			\end{itemize}
%		\item $\forall x \in S$ such that $\exists y \in S\setminus\curly{x}$, $y \notin \desc[T]{x}$ and $\exists x,y\in S$ such that $x \not\perp_T y$. (\CASE{4}, dealt in Proposition \ref{obs:gamma_4})
%	\end{itemize}
%\end{itemize}
%\end{proof}
\begin{prop}
	\label{obs:gamma_1}
	Let $T$ be a rooted spanning tree and $S  = \curly{x_1,x_2,\ldots,x_k}\subset V \setminus \curly {r_T}$ such that $|S| = k \geq 3$.  If no child-ancestor pair exists in $S$ i.e. $\forall x_i,x_j \in S$, $x_i \perp_T x_j$ , then $ \gamma(\desc[T]x_1,\desc[T]x_2,\ldots,\desc[T]x_k) = 0$.
\end{prop}
\begin{figure}[h]
	\centering
	\tikzset{every picture/.style={line width=0.75pt}} %set default line width to 0.75pt        

\begin{tikzpicture}[x=0.75pt,y=0.75pt,yscale=-1,xscale=1]
%uncomment if require: \path (0,139); %set diagram left start at 0, and has height of 139

%Straight Lines [id:da6528145492428679] 
\draw    (138.46,20.71) -- (87.75,84) ;
%Flowchart: Connector [id:dp7175624535296194] 
\draw  [fill={rgb, 255:red, 0; green, 0; blue, 0 }  ,fill opacity=1 ] (134.75,21.71) .. controls (134.75,20.22) and (135.97,19) .. (137.46,19) .. controls (138.96,19) and (140.18,20.22) .. (140.18,21.71) .. controls (140.18,23.21) and (138.96,24.43) .. (137.46,24.43) .. controls (135.97,24.43) and (134.75,23.21) .. (134.75,21.71) -- cycle ;
%Straight Lines [id:da05795274087577407] 
\draw    (113.11,52.36) -- (136.59,85.36) ;
%Shape: Triangle [id:dp756674925308974] 
\draw   (87.75,84) -- (108.75,118) -- (66.75,118) -- cycle ;
%Shape: Triangle [id:dp9176257052520445] 
\draw   (136.59,85.36) -- (156.5,117) -- (116.68,117) -- cycle ;
%Straight Lines [id:da45454638971772465] 
\draw [fill={rgb, 255:red, 155; green, 155; blue, 155 }  ,fill opacity=1 ]   (135.75,20.71) -- (191.32,85.71) ;
%Shape: Triangle [id:dp6669777984270882] 
\draw   (191.32,85.71) -- (213.05,117) -- (169.59,117) -- cycle ;
%Flowchart: Connector [id:dp1713630126939223] 
\draw  [fill={rgb, 255:red, 0; green, 0; blue, 0 }  ,fill opacity=1 ] (110.75,52.71) .. controls (110.75,51.22) and (111.97,50) .. (113.46,50) .. controls (114.96,50) and (116.18,51.22) .. (116.18,52.71) .. controls (116.18,54.21) and (114.96,55.43) .. (113.46,55.43) .. controls (111.97,55.43) and (110.75,54.21) .. (110.75,52.71) -- cycle ;
%Flowchart: Connector [id:dp012239722217227467] 
\draw  [fill={rgb, 255:red, 0; green, 0; blue, 0 }  ,fill opacity=1 ] (84.75,83.71) .. controls (84.75,82.22) and (85.97,81) .. (87.46,81) .. controls (88.96,81) and (90.18,82.22) .. (90.18,83.71) .. controls (90.18,85.21) and (88.96,86.43) .. (87.46,86.43) .. controls (85.97,86.43) and (84.75,85.21) .. (84.75,83.71) -- cycle ;
%Flowchart: Connector [id:dp938752446918804] 
\draw  [fill={rgb, 255:red, 0; green, 0; blue, 0 }  ,fill opacity=1 ] (133.75,85.75) .. controls (133.73,84.25) and (134.93,83.02) .. (136.43,83) .. controls (137.93,82.98) and (139.16,84.18) .. (139.18,85.68) .. controls (139.2,87.18) and (138,88.41) .. (136.5,88.43) .. controls (135,88.45) and (133.77,87.25) .. (133.75,85.75) -- cycle ;
%Flowchart: Connector [id:dp9399395977242846] 
\draw  [fill={rgb, 255:red, 0; green, 0; blue, 0 }  ,fill opacity=1 ] (188.75,85.71) .. controls (188.75,84.22) and (189.97,83) .. (191.46,83) .. controls (192.96,83) and (194.18,84.22) .. (194.18,85.71) .. controls (194.18,87.21) and (192.96,88.43) .. (191.46,88.43) .. controls (189.97,88.43) and (188.75,87.21) .. (188.75,85.71) -- cycle ;

% Text Node
\draw (142,5) node [anchor=north west][inner sep=0.75pt]    {$r_{T}$};
% Text Node
\draw (65,63) node [anchor=north west][inner sep=0.75pt]    {$x_{1}$};
% Text Node
\draw (130.75,62.71) node [anchor=north west][inner sep=0.75pt]    {$x_{2}$};
% Text Node
\draw (190,58) node [anchor=north west][inner sep=0.75pt]    {$x_{k}$};
% Text Node
\draw (76,95) node [anchor=north west][inner sep=0.75pt]  [font=\scriptsize]  {${\textstyle x_{1}^{\downarrow T}}$};
% Text Node
\draw (124,95) node [anchor=north west][inner sep=0.75pt]  [font=\scriptsize]  {${\textstyle x_{2}^{\downarrow T}}$};
% Text Node
\draw (181,96) node [anchor=north west][inner sep=0.75pt]  [font=\scriptsize]  {${\textstyle x_{k}^{\downarrow T}}$};

\end{tikzpicture}
	\caption{Figure shows the orientation of $x_1,x_2,\ldots,x_k$ in terms of their ancestor-decedent relationships for \CASE{1}}
\end{figure}
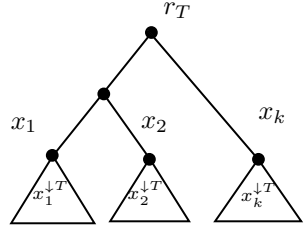
\begin{proof}
	Here, $\forall x_i,x_j\in S$, we have $x_i \perp_T x_j$. This implies $\desc[T]{x_i} \cap \desc[T]{x_j}  = \emptyset$. Since, $|S| \geq 3$, no edge can have endpoints in $\desc[T]{x_i}$ for all $i$ simultaneously (an edge has only two endpoints). Thus $\gamma(\desc[T]x_1,\desc[T]x_2,\ldots,\desc[T]x_k) = 0$.
\end{proof}
\begin{prop}
	\label{obs:gamma_2}
	Let $T$ be a rooted spanning tree and $S  = \curly{x_1,x_2,\ldots,x_k}\subset V \setminus \curly {r_T}$ such that $|S| = k \geq 3$. If 
	\begin{enumerate}
		\item $\exists x,y\in S, x\not\perp_T y$ (exists two nodes that are not independent),
		\item $\exists x'\in S,\ \forall y'\in S,\ y' \in \desc[T]{x'}$ (there exists an $x'$ in $S$ such that all $y'$ in $S$ are decedents of x),
		\item $\exists x''\in S,\ \forall y''\in S, y''\in x'^{\prime \uparrow T}$ (there exists an $x''$ in $S$ such that all $y''$ in $S$ are ancestors of $x$).
	\end{enumerate}
Then $\gamma(\desc[T]{x_1},\desc[T]{x_2},\ldots,\desc[T]x_{k}) = \gamma(\desc[T]{p},\desc[T]{q})$, where $p = \argmax_{x \in S} \level[T]{x}$ and $q = \argmin_{x \in S}\level[T]{x}$.
\end{prop}	
\begin{figure}[h]
	\centering
	\tikzset{every picture/.style={line width=0.75pt}} %set default line width to 0.75pt        

\begin{tikzpicture}[x=0.75pt,y=0.75pt,yscale=-1,xscale=1]
%uncomment if require: \path (0,211); %set diagram left start at 0, and has height of 211

%Straight Lines [id:da7175216397380968] 
\draw    (173.43,12.57) -- (180.46,122.71) ;
%Flowchart: Connector [id:dp08441673498106073] 
\draw  [fill={rgb, 255:red, 0; green, 0; blue, 0 }  ,fill opacity=1 ] (189.75,-9.29) .. controls (189.75,-10.78) and (190.97,-12) .. (192.46,-12) .. controls (193.96,-12) and (195.18,-10.78) .. (195.18,-9.29) .. controls (195.18,-7.79) and (193.96,-6.57) .. (192.46,-6.57) .. controls (190.97,-6.57) and (189.75,-7.79) .. (189.75,-9.29) -- cycle ;
%Shape: Triangle [id:dp9735893667376674] 
\draw   (180.75,123) -- (201.75,157) -- (159.75,157) -- cycle ;
%Shape: Triangle [id:dp05361766591216499] 
\draw   (178.43,84.32) -- (216.43,169.57) -- (140.43,169.57) -- cycle ;
%Shape: Triangle [id:dp41734957546234686] 
\draw   (174.93,34.57) -- (260.43,201.57) -- (108.43,201.57) -- cycle ;
%Flowchart: Connector [id:dp5966985717949549] 
\draw  [fill={rgb, 255:red, 0; green, 0; blue, 0 }  ,fill opacity=1 ] (172.71,37.29) .. controls (172.71,35.79) and (173.93,34.57) .. (175.43,34.57) .. controls (176.93,34.57) and (178.14,35.79) .. (178.14,37.29) .. controls (178.14,38.78) and (176.93,40) .. (175.43,40) .. controls (173.93,40) and (172.71,38.78) .. (172.71,37.29) -- cycle ;
%Flowchart: Connector [id:dp9956915438879868] 
\draw  [fill={rgb, 255:red, 0; green, 0; blue, 0 }  ,fill opacity=1 ] (177.75,122.71) .. controls (177.75,121.22) and (178.97,120) .. (180.46,120) .. controls (181.96,120) and (183.18,121.22) .. (183.18,122.71) .. controls (183.18,124.21) and (181.96,125.43) .. (180.46,125.43) .. controls (178.97,125.43) and (177.75,124.21) .. (177.75,122.71) -- cycle ;
%Flowchart: Connector [id:dp6139162755811616] 
\draw  [fill={rgb, 255:red, 0; green, 0; blue, 0 }  ,fill opacity=1 ] (175.75,87.06) .. controls (175.73,85.56) and (176.93,84.33) .. (178.43,84.32) .. controls (179.93,84.3) and (181.16,85.5) .. (181.18,87) .. controls (181.19,88.5) and (179.99,89.73) .. (178.49,89.74) .. controls (176.99,89.76) and (175.77,88.56) .. (175.75,87.06) -- cycle ;
%Flowchart: Connector [id:dp5838006777798903] 
\draw  [fill={rgb, 255:red, 0; green, 0; blue, 0 }  ,fill opacity=1 ] (170.71,14.29) .. controls (170.71,12.79) and (171.93,11.57) .. (173.43,11.57) .. controls (174.93,11.57) and (176.14,12.79) .. (176.14,14.29) .. controls (176.14,15.78) and (174.93,17) .. (173.43,17) .. controls (171.93,17) and (170.71,15.78) .. (170.71,14.29) -- cycle ;
%Flowchart: Connector [id:dp5356169510739881] 
\draw  [fill={rgb, 255:red, 0; green, 0; blue, 0 }  ,fill opacity=1 ] (173.71,52.29) .. controls (173.71,50.79) and (174.93,49.57) .. (176.43,49.57) .. controls (177.93,49.57) and (179.14,50.79) .. (179.14,52.29) .. controls (179.14,53.78) and (177.93,55) .. (176.43,55) .. controls (174.93,55) and (173.71,53.78) .. (173.71,52.29) -- cycle ;
%Flowchart: Connector [id:dp7620312695200724] 
\draw  [fill={rgb, 255:red, 0; green, 0; blue, 0 }  ,fill opacity=1 ] (174.71,69.29) .. controls (174.71,67.79) and (175.93,66.57) .. (177.43,66.57) .. controls (178.93,66.57) and (180.14,67.79) .. (180.14,69.29) .. controls (180.14,70.78) and (178.93,72) .. (177.43,72) .. controls (175.93,72) and (174.71,70.78) .. (174.71,69.29) -- cycle ;

% Text Node
\draw (152,4.4) node [anchor=north west][inner sep=0.75pt]    {$r_{T}$};
% Text Node
\draw (183,112.4) node [anchor=north west][inner sep=0.75pt]  [font=\small]  {$x_{1}$};
% Text Node
\draw (178.95,71.04) node [anchor=north west][inner sep=0.75pt]  [font=\small]  {$x_{2}$};
% Text Node
\draw (182,25.4) node [anchor=north west][inner sep=0.75pt]    {$x_{k}$};
% Text Node
\draw (171,137.4) node [anchor=north west][inner sep=0.75pt]  [font=\tiny]  {${\textstyle x_{1}^{\downarrow T}}$};
% Text Node
\draw (161,122.4) node [anchor=north west][inner sep=0.75pt]  [font=\tiny]  {${\textstyle x_{2}^{\downarrow T}}$};
% Text Node
\draw (236,185.4) node [anchor=north west][inner sep=0.75pt]  [font=\scriptsize]  {${\textstyle x_{k}^{\downarrow T}}$};

\end{tikzpicture}
		\caption{Figure shows the orientation of $x_1,x_2,\ldots,x_k$ in terms of their ancestor-decedent relationships for \CASE{2}}
\end{figure}
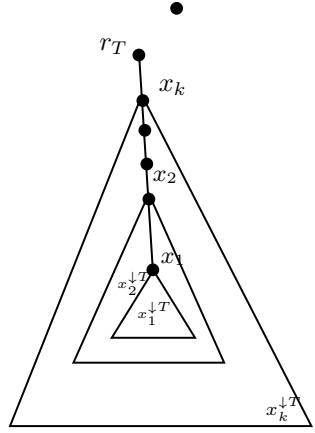
\begin{proof}
	Here, the third condition, subsumes the other two conditions. This is because, if there exists a node $x''$ such that all other nodes are ancestors of it (they are on the path from $x''$ to root $r_T$). Then firstly, all of these are independent. Secondly, there also exists a node such that all the other nodes are decedents of such a node.
	
	WLOG, let $\level[T]{x_1} > \level[T]{x_2} > \ldots > \level[T]{x_k} $.
	Since all vertices in $S$ are on the tree path from $r_T$ to $x_1$, 
	thus we have, $\desc[T]{x_1} \subset \desc[T]{x_2} \ldots \subset \desc[T]{x_k}$. Similarly, $V\setminus \desc[T]{x_1} \supset V\setminus \desc[T]{x_2} \ldots \supset V\setminus \desc[T]{x_k}$.
	By definition, $\delta(\desc[T]{x}) = \curly{(u,v) | u \in \desc[T]{x}, v\in V\setminus \desc[T]{x} }$. Hence, 
	\begin{align}
	&\delta(\desc[T]{x_1}) \cap \cdots \cap \delta(\desc[T]{x_k}) 
	\nonumber\\=& \curly{(u,v) | u \in \desc[T]{x_1}, v\in \bigcap_i V\setminus \desc[T]{x_1}} \bigcap \cdots \bigcap \curly{(u,v) | u \in \desc[T]{x_k}, v\in \bigcap_i V\setminus \desc[T]{x_k}}\nonumber\\ 
	=& \curly{(u,v) | u \in \desc[T]{x_1}, v\in V\setminus \desc[T]{x_k}}\quad (\because \desc[T]{x_1} \subset \desc[T]{x_2} \ldots \subset \desc[T]{x_k})\nonumber\\ 
	=&\curly{(u',v') | u' \in \desc[T]{x_1}, v'\in V\setminus \desc[T]{x_1}} \cap \curly{(u'',v'') | u'' \in \desc[T]{x_k}, v''\in V\setminus \desc[T]{x_k}}\nonumber\\
	=& \delta(\desc[T]{x_1}) \cap \delta(\desc[T]{x_k})\label{obs-gamma-2-}
	\end{align}		 

 Hence, $\gamma(\desc[T]{x_1},\desc[T]{x_2},\ldots,\desc[T]x_{k}) = |\delta(\desc[T]{x_1}) \cap \cdots \cap \delta(\desc[T]{x_k}) | = |\delta(\desc[T]{x_1}) \cap \delta(\desc[T]{x_k})| = \gamma(\desc[T]{x_1},\desc[T]{x_k}) =\gamma(\desc[T]{p},\desc[T]{q})$, where $p = \argmax_{x \in S} \level[T]{x}$ and $q = \argmin_{x \in S}\level[T]{x}$.
\end{proof}
\begin{prop}
	\label{obs:gamma_3}
	Let $T$ be a rooted spanning tree and $S  = \curly{x_1,x_2,\ldots,x_k}\subset V \setminus \curly {r_T}$ such that $|S| = k \geq 3$. If 
	\begin{enumerate}
		\item $\exists x,y\in S, x\not\perp_T y$ (exists two vertices that are not independent),
		\item $\exists x'\in S,\ \forall y'\in S,\ y' \in \desc[T]{x'}$ (there exists an $x'$ in $S$ such that all $y'$ in $S$ are decedents of x), 
		\item $\forall x''\in S,\ \exists y''\in S, y''\not\in x'^{\prime \uparrow T}$ (for all $x''$ in $S$ there exists a $y''$ in $S$ such that $y''$ is not ancestor of $x''$).
	\end{enumerate}
	Then $\gamma(\desc[T]{x_1},\desc[T]{x_2},\ldots,\desc[T]x_{k}) = 0$.
\end{prop}
\begin{figure}[h]
	\centering
	\tikzset{every picture/.style={line width=0.75pt}} %set default line width to 0.75pt        

\begin{tikzpicture}[x=0.75pt,y=0.75pt,yscale=-1,xscale=1]
%uncomment if require: \path (0,174); %set diagram left start at 0, and has height of 174

%Straight Lines [id:da7175216397380968] 
\draw    (173.43,12.57) -- (178.49,89.74) ;
%Flowchart: Connector [id:dp08441673498106073] 
\draw  [fill={rgb, 255:red, 0; green, 0; blue, 0 }  ,fill opacity=1 ] (189.75,-9.29) .. controls (189.75,-10.78) and (190.97,-12) .. (192.46,-12) .. controls (193.96,-12) and (195.18,-10.78) .. (195.18,-9.29) .. controls (195.18,-7.79) and (193.96,-6.57) .. (192.46,-6.57) .. controls (190.97,-6.57) and (189.75,-7.79) .. (189.75,-9.29) -- cycle ;
%Shape: Triangle [id:dp9735893667376674] 
\draw   (147.38,120.71) -- (166,148.67) -- (128.75,148.67) -- cycle ;
%Shape: Triangle [id:dp41734957546234686] 
\draw   (175.43,37.29) -- (289.8,158.71) -- (86.48,158.71) -- cycle ;
%Flowchart: Connector [id:dp5966985717949549] 
\draw  [fill={rgb, 255:red, 0; green, 0; blue, 0 }  ,fill opacity=1 ] (172.71,37.29) .. controls (172.71,35.79) and (173.93,34.57) .. (175.43,34.57) .. controls (176.93,34.57) and (178.14,35.79) .. (178.14,37.29) .. controls (178.14,38.78) and (176.93,40) .. (175.43,40) .. controls (173.93,40) and (172.71,38.78) .. (172.71,37.29) -- cycle ;
%Flowchart: Connector [id:dp9956915438879868] 
\draw  [fill={rgb, 255:red, 0; green, 0; blue, 0 }  ,fill opacity=1 ] (145.04,120.71) .. controls (145.04,119.22) and (146.25,118) .. (147.75,118) .. controls (149.25,118) and (150.46,119.22) .. (150.46,120.71) .. controls (150.46,122.21) and (149.25,123.43) .. (147.75,123.43) .. controls (146.25,123.43) and (145.04,122.21) .. (145.04,120.71) -- cycle ;
%Flowchart: Connector [id:dp6139162755811616] 
\draw  [fill={rgb, 255:red, 0; green, 0; blue, 0 }  ,fill opacity=1 ] (175.75,87.06) .. controls (175.73,85.56) and (176.93,84.33) .. (178.43,84.32) .. controls (179.93,84.3) and (181.16,85.5) .. (181.18,87) .. controls (181.19,88.5) and (179.99,89.73) .. (178.49,89.74) .. controls (176.99,89.76) and (175.77,88.56) .. (175.75,87.06) -- cycle ;
%Flowchart: Connector [id:dp5838006777798903] 
\draw  [fill={rgb, 255:red, 0; green, 0; blue, 0 }  ,fill opacity=1 ] (170.71,14.29) .. controls (170.71,12.79) and (171.93,11.57) .. (173.43,11.57) .. controls (174.93,11.57) and (176.14,12.79) .. (176.14,14.29) .. controls (176.14,15.78) and (174.93,17) .. (173.43,17) .. controls (171.93,17) and (170.71,15.78) .. (170.71,14.29) -- cycle ;
%Flowchart: Connector [id:dp5356169510739881] 
\draw  [fill={rgb, 255:red, 0; green, 0; blue, 0 }  ,fill opacity=1 ] (200.75,105) .. controls (200.75,103.5) and (201.97,102.29) .. (203.46,102.29) .. controls (204.96,102.29) and (206.18,103.5) .. (206.18,105) .. controls (206.18,106.5) and (204.96,107.71) .. (203.46,107.71) .. controls (201.97,107.71) and (200.75,106.5) .. (200.75,105) -- cycle ;
%Straight Lines [id:da5521488620333013] 
\draw    (178.46,87.03) -- (147.75,120.71) ;
%Shape: Triangle [id:dp8910564464888047] 
\draw   (202.75,105) -- (223.75,139) -- (181.75,139) -- cycle ;
%Straight Lines [id:da6787186828118168] 
\draw    (178.46,87.03) -- (202.75,105) ;

% Text Node
\draw (152,4.4) node [anchor=north west][inner sep=0.75pt]    {$r_{T}$};
% Text Node
\draw (144,100.4) node [anchor=north west][inner sep=0.75pt]  [font=\small]  {$x_{1}$};
% Text Node
\draw (205.95,90.04) node [anchor=north west][inner sep=0.75pt]  [font=\small]  {$x_{2}$};
% Text Node
\draw (182,25.4) node [anchor=north west][inner sep=0.75pt]    {$x_{k}$};
% Text Node
\draw (138,132.4) node [anchor=north west][inner sep=0.75pt]  [font=\tiny]  {${\textstyle x_{1}^{\downarrow T}}$};
% Text Node
\draw (196,126.4) node [anchor=north west][inner sep=0.75pt]  [font=\tiny]  {${\textstyle x_{2}^{\downarrow T}}$};

\end{tikzpicture}
	\caption{Figure shows the orientation of $x_1,x_2,\ldots,x_k$ in terms of their ancestor-decedent relationships for \CASE{3}}
\end{figure}
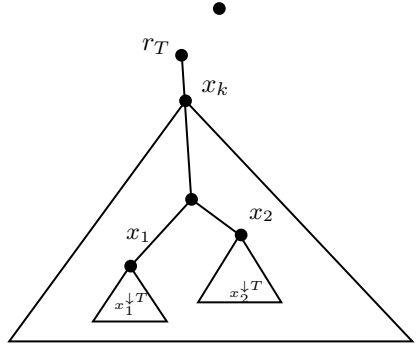
\begin{proof}
Here, condition (2) subsumes condition (1). This is because there exists a node such that all other nodes are its decedents. Let $x_k$ be such a node. Hence $x_k\not\perp_T x_i$ for all $x_i \in S$, and $x_i \neq x_k$.
Let $x_1 = \argmax_{x \in S}$ be the node that is furthest from the root $r_T$ in the tree $T$. If many such nodes exist, then choose an arbitrary node. Choosing a furthest node implies that there does not exist $x \in S$, such that $x \in \desc[T]{x_1}$. By condition (3), we know that there is a node $x_2\notin x_1^{\uparrow T}$. But $x_2 \in \desc[T]x_k$ (condition (2)), and by choice of $x_1$, $x_2 \notin\desc[T]{x_1}$. Thus the set relationships of $\desc[T]{x_1}$, $\desc[T]{x_2}$, and $\desc[T]{x_k}$ is as follows: $\desc[T]{x_1} \subset \desc[T]{x_k}$, $\desc[T]{x_2} \subset \desc[T]{x_k}$ and $\desc[T]{x_1} \cap \desc[T]{x_2} = \emptyset$. 

In the remaining part we show $\delta(\desc[T]{x_1}) \cap \delta(\desc[T]{x_2}) \cap \delta(\desc[T]{x_k}) = \emptyset$. Firstly, $\delta(\desc[T]{x_1}) \cap \delta(\desc[T]{x_2})$ contains those edges which have one end point in $\desc[T]{x_1}$ and other in $\desc[T]{x_2}$. Also, $\delta(\desc[T]{x_k})$ contains those edges which have one end point in $\desc[T]{x_k}$ and the other point outside of $\desc[T]{x_k}$. Because of the aforementioned set relationship, no edge in $\delta(\desc[T]{x_k})$ is contained in $\delta(\desc[T]{x_1}) \cap \delta(\desc[T]{x_2})$. Hence $\delta(\desc[T]{x_1}) \cap \delta(\desc[T]{x_2}) \cap \delta(\desc[T]{x_k}) = \emptyset$. Thus $\delta(\desc[T]{x_1}) \cap \delta(\desc[T]{x_2}) \cap \cdots \cap \delta(\desc[T]{x_k}) = \emptyset$ which implies that $\gamma(\desc[T]x_1,\desc[T]x_2,\ldots,\desc[T]x_k) = 0$.
\end{proof}
\begin{prop}
	\label{obs:gamma_4}
	Let $T$ be a rooted spanning tree and $S  = \curly{x_1,x_2,\ldots,x_k}\subset V \setminus \curly {r_T}$ such that $|S| = k \geq 3$. If 
	\begin{enumerate}
		\item $\exists x,y\in S, x\not\perp_T y$ (exists two vertices that are not independent),
		\item $\forall x' \in S, \exists y' \in S, y' \notin \desc[T]{x'}$
	\end{enumerate}
	Then $\gamma(\desc[T]{x_1},\desc[T]{x_2},\ldots,\desc[T]x_{k}) = \gamma(\desc[T]{y_1},\ldots,\desc[T]y_{k-1})$ where $\curly{y_1,y_2,\ldots,y_{k-1}} = S \setminus \curly{a}$ for some $a \in S$.
\end{prop}
\begin{figure}[h]
	\centering
	\tikzset{every picture/.style={line width=0.75pt}} %set default line width to 0.75pt        

\begin{tikzpicture}[x=0.75pt,y=0.75pt,yscale=-1,xscale=1]
%uncomment if require: \path (0,150); %set diagram left start at 0, and has height of 150

%Straight Lines [id:da30495313637386556] 
\draw    (152.43,18.14) -- (155.43,99.71) ;
%Shape: Triangle [id:dp11418274808024198] 
\draw   (154.38,102.71) -- (173,130.67) -- (135.75,130.67) -- cycle ;
%Shape: Triangle [id:dp0899567921537503] 
\draw   (155.03,72.07) -- (212.21,138.5) -- (108.08,138.5) -- cycle ;
%Flowchart: Connector [id:dp8091883385105711] 
\draw  [fill={rgb, 255:red, 0; green, 0; blue, 0 }  ,fill opacity=1 ] (152.71,73.29) .. controls (152.71,71.79) and (153.93,70.57) .. (155.43,70.57) .. controls (156.93,70.57) and (158.14,71.79) .. (158.14,73.29) .. controls (158.14,74.78) and (156.93,76) .. (155.43,76) .. controls (153.93,76) and (152.71,74.78) .. (152.71,73.29) -- cycle ;
%Flowchart: Connector [id:dp7962430808536638] 
\draw  [fill={rgb, 255:red, 0; green, 0; blue, 0 }  ,fill opacity=1 ] (152.04,102.71) .. controls (152.04,101.22) and (153.25,100) .. (154.75,100) .. controls (156.25,100) and (157.46,101.22) .. (157.46,102.71) .. controls (157.46,104.21) and (156.25,105.43) .. (154.75,105.43) .. controls (153.25,105.43) and (152.04,104.21) .. (152.04,102.71) -- cycle ;
%Flowchart: Connector [id:dp6601926898329162] 
\draw  [fill={rgb, 255:red, 0; green, 0; blue, 0 }  ,fill opacity=1 ] (352.75,66.06) .. controls (352.73,64.56) and (353.93,63.33) .. (355.43,63.32) .. controls (356.93,63.3) and (358.16,64.5) .. (358.18,66) .. controls (358.19,67.5) and (356.99,68.73) .. (355.49,68.74) .. controls (353.99,68.76) and (352.77,67.56) .. (352.75,66.06) -- cycle ;
%Flowchart: Connector [id:dp06765955105748023] 
\draw  [fill={rgb, 255:red, 0; green, 0; blue, 0 }  ,fill opacity=1 ] (149.71,15.43) .. controls (149.71,13.93) and (150.93,12.71) .. (152.43,12.71) .. controls (153.93,12.71) and (155.14,13.93) .. (155.14,15.43) .. controls (155.14,16.93) and (153.93,18.14) .. (152.43,18.14) .. controls (150.93,18.14) and (149.71,16.93) .. (149.71,15.43) -- cycle ;
%Flowchart: Connector [id:dp48400280052651157] 
\draw  [fill={rgb, 255:red, 0; green, 0; blue, 0 }  ,fill opacity=1 ] (243.04,86) .. controls (243.04,84.5) and (244.25,83.29) .. (245.75,83.29) .. controls (247.25,83.29) and (248.46,84.5) .. (248.46,86) .. controls (248.46,87.5) and (247.25,88.71) .. (245.75,88.71) .. controls (244.25,88.71) and (243.04,87.5) .. (243.04,86) -- cycle ;
%Straight Lines [id:da1766332720914836] 
\draw    (153.43,37.71) -- (245.75,86) ;
%Shape: Triangle [id:dp21592871918889966] 
\draw   (245.75,86) -- (266.75,120) -- (224.75,120) -- cycle ;

% Text Node
\draw (129,6.4) node [anchor=north west][inner sep=0.75pt]    {$r_{T}$};
% Text Node
\draw (159,88.4) node [anchor=north west][inner sep=0.75pt]  [font=\footnotesize]  {$v$};
% Text Node
\draw (245.95,67.04) node [anchor=north west][inner sep=0.75pt]  [font=\small]  {$u$};
% Text Node
\draw (139,58.4) node [anchor=north west][inner sep=0.75pt]    {$a$};
% Text Node
\draw (145,114.4) node [anchor=north west][inner sep=0.75pt]  [font=\tiny]  {${\textstyle v^{\downarrow T}}$};
% Text Node
\draw (241,103.4) node [anchor=north west][inner sep=0.75pt]  [font=\tiny]  {${\textstyle u^{\downarrow T}}$};
% Text Node
\draw (183,123.4) node [anchor=north west][inner sep=0.75pt]  [font=\tiny]  {${\textstyle a^{\downarrow T}}$};

\end{tikzpicture}
	\caption{Figure shows the orientation of $a,v,u$ in terms of their ancestor-decedent relationships for \CASE{4}}
\end{figure}
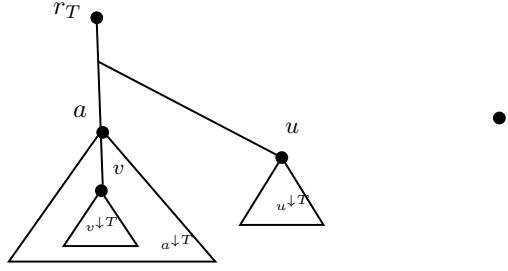
\begin{proof}
Let $a,v$ be two nodes such that $a\not\perp_T v$ (condition (1)), where $\level[T]{a} < \level[T]{v}$. Here we choose an $a,v$ pair such that no other node $a'$ exists in $S$ such that $\level[T]{a'} < \level[T]{a}$ and $a' \not\perp_T a$. If it exists then our chosen pair is $a',a$. Based on condition (3), we know that $\exists u$ such that $u \notin \desc[T]{a}$. According to the choice of $u,v$, and $a$, we show that $\desc[T]{v} \subset \desc[T]{a}$, $a \perp_T u$, and $v\perp_T u$. Firstly, $\desc[T]{v} \subset \desc[T]{a}$ is true because $a\not\perp_T v$, and $\level[T]{a} < \level[T]{v}$ (see Proposition \ref{prop-simple}). From the choice, $u \notin \desc[T]{a}$. Then either $a \perp_T u$, or $u$ is the ancestor of $a$. The latter cannot be true because we have chosen $a$ such that no $a' \in S$ exists such that $\level[T]{a'} < \level[T]{a}$ and $a'\not\perp_T a$. Thus $u \perp_T a$. This also implies that $u\perp_T v$, because $\desc[T]v \subset \desc[T]{a}$.

Lastly, in this proof we show that $\delta(\desc[T]a) \cap \delta(\desc[T]u) \cap \delta(\desc[T]v) =\delta(\desc[T]u) \cap \delta(\desc[T]v)$. This implies that $\gamma(\desc[T]{x_1},\desc[T]{x_2}\ldots,\desc[T]x_{k}) = \gamma(\desc[T]{y_1},\desc[T]{y_2}\ldots,\desc[T]y_{k-1})$, where $\curly{y_1,y_2,\ldots,y_{k-1}} = Q \setminus \curly{a}$. That is, we can eliminate $\desc[T]{a}$ to compute \sloppy$\gamma(\desc[T]{x_1},\desc[T]{x_2}\ldots,\desc[T]x_{k})$. 
\begin{align*}
&\delta(\desc[T]a) \cap \delta(\desc[T]v) \cap \delta(\desc[T]u) \\
=&\curly{(x,y) | x\in \desc[T]{a}, y\in V\setminus \desc[T]{a}} 
\cap \curly{(x,y) | x\in \desc[T]{v}, y\in V\setminus \desc[T]{v}}\\ 
&\cap \curly{(x,y) | x\in \desc[T]{u}, y\in V\setminus \desc[T]{u}}\\
=&\curly{(x,y) | x\in \desc[T]{v}, y\in V\setminus \desc[T]{a}} \cap \curly{(x,y) | x\in \desc[T]{u}, y\in V\setminus \desc[T]{u}}\quad (\therefore \desc[T]{v} \subset \desc[T]{a}) \\
=&\curly{(x,y) | x\in \desc[T]{v}, y\in \desc[T]{u}}\quad (\because \desc[T]{v} \subset V\setminus \desc[T]{u} \text{\ and\ }  \desc[T]{u} \subset V\setminus \desc[T]{a}) \\
=&\delta\paren{\desc[T]{u}} \cap \delta\paren{\desc[T]{v}}
\end{align*}
The last inequality is true because $u\perp_T v$ and $\curly{(x,y) | x\in \desc[T]{v}, y\in \desc[T]{u}}$ {contain all the edges} that have an endpoint in $\desc[T]{u}$ and $\desc[T]{v}$
\end{proof}

\nocite{karger2000minimum}
\nocite{thorup2007fully}

\bibliography{mybibfile.bib}

\newcommand{\etalchar}[1]{$^{#1}$}
\begin{thebibliography}{CMW{\etalchar{+}}94}

\bibitem[AKL{\etalchar{+}}21]{Abboud2021-cd}
Amir Abboud, Robert Krauthgamer, Jason Li, Debmalya Panigrahi, Thatchaphol
  Saranurak, and Ohad Trabelsi.
\newblock Breaking the cubic barrier for all-pairs max-flow: {Gomory-Hu} tree
  in nearly quadratic time.
\newblock November 2021.

\bibitem[BLS20]{bhardwaj2020simple}
Nalin Bhardwaj, Antonio~Molina Lovett, and Bryce Sandlund.
\newblock A simple algorithm for minimum cuts in near-linear time.
\newblock In Susanne Albers, editor, {\em 17th Scandinavian Symposium and
  Workshops on Algorithm Theory, {SWAT} 2020, June 22-24, 2020, T{\'{o}}rshavn,
  Faroe Islands}, volume 162 of {\em LIPIcs}, pages 12:1--12:18. Schloss
  Dagstuhl - Leibniz-Zentrum f{\"{u}}r Informatik, 2020.

\bibitem[BM76]{bondy1976graph}
J.~Adrian Bondy and Uppaluri S.~R. Murty.
\newblock {\em Graph Theory with Applications}.
\newblock Macmillan Education {UK}, 1976.

\bibitem[CMW{\etalchar{+}}94]{chen1994short}
Boliong Chen, Makoto Matsumoto, Jianfang Wang, Zhongfu Zhang, and Jianxun
  Zhang.
\newblock A short proof of nash-williams' theorem for the arboricity of a
  graph.
\newblock {\em Graphs Comb.}, 10(1):27--28, 1994.

\bibitem[DEMN21]{dory2021distributed}
Michal Dory, Yuval Efron, Sagnik Mukhopadhyay, and Danupon Nanongkai.
\newblock Distributed weighted min-cut in nearly-optimal time.
\newblock In Samir Khuller and Virginia~Vassilevska Williams, editors, {\em
  {STOC} '21: 53rd Annual {ACM} {SIGACT} Symposium on Theory of Computing,
  Virtual Event, Italy, June 21-25, 2021}, pages 1144--1153. {ACM}, 2021.

\bibitem[DHNS19]{daga2019distributed}
Mohit Daga, Monika Henzinger, Danupon Nanongkai, and Thatchaphol Saranurak.
\newblock Distributed edge connectivity in sublinear time.
\newblock In Moses Charikar and Edith Cohen, editors, {\em Proceedings of the
  51st Annual {ACM} {SIGACT} Symposium on Theory of Computing, {STOC} 2019,
  Phoenix, AZ, USA, June 23-26, 2019}, pages 343--354. {ACM}, 2019.

\bibitem[Die12]{diestel2018graph}
Reinhard Diestel.
\newblock {\em Graph Theory, 4th Edition}, volume 173 of {\em Graduate texts in
  mathematics}.
\newblock Springer, 2012.

\bibitem[GMW20]{gawrychowski2020minimum}
Pawel Gawrychowski, Shay Mozes, and Oren Weimann.
\newblock Minimum cut in o(m log{\({^2}\)} n) time.
\newblock In Artur Czumaj, Anuj Dawar, and Emanuela Merelli, editors, {\em 47th
  International Colloquium on Automata, Languages, and Programming, {ICALP}
  2020, July 8-11, 2020, Saarbr{\"{u}}cken, Germany (Virtual Conference)},
  volume 168 of {\em LIPIcs}, pages 57:1--57:15. Schloss Dagstuhl -
  Leibniz-Zentrum f{\"{u}}r Informatik, 2020.

\bibitem[GNT20]{ghaffari2020faster}
Mohsen Ghaffari, Krzysztof Nowicki, and Mikkel Thorup.
\newblock Faster algorithms for edge connectivity via random 2-out
  contractions.
\newblock In Shuchi Chawla, editor, {\em Proceedings of the 2020 {ACM-SIAM}
  Symposium on Discrete Algorithms, {SODA} 2020, Salt Lake City, UT, USA,
  January 5-8, 2020}, pages 1260--1279. {SIAM}, 2020.

\bibitem[GZ22]{ghaffari2022universally}
Mohsen Ghaffari and Goran Zuzic.
\newblock Universally-optimal distributed exact min-cut.
\newblock pages 281--291, 2022.

\bibitem[HRW20]{henzinger2020local}
Monika Henzinger, Satish Rao, and Di~Wang.
\newblock Local flow partitioning for faster edge connectivity.
\newblock {\em {SIAM} J. Comput.}, 49(1):1--36, 2020.

\bibitem[Kar00]{karger2000minimum}
David~R. Karger.
\newblock Minimum cuts in near-linear time.
\newblock {\em J. {ACM}}, 47(1):46--76, 2000.

\bibitem[KT15]{kawarabayashi2015deterministic}
Ken{-}ichi Kawarabayashi and Mikkel Thorup.
\newblock Deterministic global minimum cut of a simple graph in near-linear
  time.
\newblock In Rocco~A. Servedio and Ronitt Rubinfeld, editors, {\em Proceedings
  of the Forty-Seventh Annual {ACM} on Symposium on Theory of Computing, {STOC}
  2015, Portland, OR, USA, June 14-17, 2015}, pages 665--674. {ACM}, 2015.

\bibitem[MN20]{mukhopadhyay2020weighted}
Sagnik Mukhopadhyay and Danupon Nanongkai.
\newblock Weighted min-cut: sequential, cut-query, and streaming algorithms.
\newblock In Konstantin Makarychev, Yury Makarychev, Madhur Tulsiani, Gautam
  Kamath, and Julia Chuzhoy, editors, {\em Proccedings of the 52nd Annual {ACM}
  {SIGACT} Symposium on Theory of Computing, {STOC} 2020, Chicago, IL, USA,
  June 22-26, 2020}, pages 496--509. {ACM}, 2020.

\bibitem[Par19]{parter2019small}
Merav Parter.
\newblock Small cuts and connectivity certificates: {A} fault tolerant
  approach.
\newblock In Jukka Suomela, editor, {\em 33rd International Symposium on
  Distributed Computing, {DISC} 2019, October 14-18, 2019, Budapest, Hungary},
  volume 146 of {\em LIPIcs}, pages 30:1--30:16. Schloss Dagstuhl -
  Leibniz-Zentrum f{\"{u}}r Informatik, 2019.

\bibitem[PT11]{pritchard2011fast}
David Pritchard and Ramakrishna Thurimella.
\newblock Fast computation of small cuts via cycle space sampling.
\newblock {\em {ACM} Trans. Algorithms}, 7(4):46:1--46:30, 2011.

\bibitem[Sar21]{saranurak2021simple}
Thatchaphol Saranurak.
\newblock A simple deterministic algorithm for edge connectivity.
\newblock In Hung~Viet Le and Valerie King, editors, {\em 4th Symposium on
  Simplicity in Algorithms, {SOSA} 2021, Virtual Conference, January 11-12,
  2021}, pages 80--85. {SIAM}, 2021.

\bibitem[Tho07]{thorup2007fully}
Mikkel Thorup.
\newblock Fully-dynamic min-cut.
\newblock {\em Comb.}, 27(1):91--127, 2007.

\end{thebibliography}

\end{document}